\documentclass{amsart}
\usepackage[foot]{amsaddr}
\usepackage{amssymb,amsmath,amsfonts,mathptmx,cite,mathrsfs,url}
\usepackage[mathscr]{eucal}

\DeclareSymbolFont{rsfscript}{OMS}{rsfs}{m}{n}
\DeclareSymbolFontAlphabet{\mathrsfs}{rsfscript}

\newtheorem{theorem}{Theorem}[section]
\newtheorem{proposition}[theorem]{Proposition}
\newtheorem{lemma}[theorem]{Lemma}
\newtheorem{corollary}[theorem]{Corollary}

\theoremstyle{remark}
\newtheorem{remark}{Remark}

\newcommand{\ais}{ai-semi\-ring}

\newcommand{\mB}{\mathcal{B}}

\newcommand{\mS}{\mathcal{S}}

\newcommand{\mY}{\mathcal{Y}}

\def\fb{finitely based}
\def\nfb{non\-finitely based}

\DeclareMathOperator{\var}{var}

\makeatletter

\renewcommand*\subjclass[2][2010]{\def\@subjclass{#2}\@ifundefined{subjclassname@#1}{\ClassWarning{\@classname}{Unknown edition (#1) of Mathematics Subject Classification; using '2010'.}}{\@xp\let\@xp\subjclassname\csname subjclassname@#1\endcsname}}

\renewcommand{\subjclassname}{\textup{2010} Mathematics Subject Classification}

\makeatother

\begin{document}

\title[Semiring and involution identities of powers of inverse semigroups]{Semiring and involution identities\\ of powers of inverse semigroups}
\thanks{I. Dolinka was supported by the Personal grant F-121 of the Serbian Academy of Sciences and Arts, and, partially, by the Ministry of Science, Technological Development and Innovations of the Republic of Serbia. S. V. Gusev and M. V. Volkov were supported by the Russian Science Foundation (grant No. 22-21-00650).}

\author[I. Dolinka]{Igor Dolinka}
\address{{\normalfont (I. Dolinka) Department of Mathematics and Informatics, University of Novi Sad, 21101 Novi Sad, Serbia}}
\email{dockie@dmi.uns.ac.rs}

\author[S. V. Gusev]{Sergey V. Gusev}
\address{{\normalfont (S. V. Gusev, M. V. Volkov) Institute of Natural Sciences and Mathematics, Ural Federal University, 620000 Ekaterinburg, Russia}}
\email{sergey.gusb@gmail.com}
\email{m.v.volkov@urfu.ru}

\author[M. V. Volkov]{Mikhail V. Volkov}

\begin{abstract}
The set of all subsets of any inverse semigroup forms an involution semiring under set-theoretical union and element-wise multiplication and inversion. We find structural conditions on a finite inverse semigroup guaranteeing that neither semiring nor involution identities of the involution semiring of its subsets admit a finite identity basis.
\end{abstract}



\maketitle

\section{Introduction}
\label{sec:introduction}

In this note, an \emph{additively idempotent semiring} (\ais, for short) is understood as an algebraic structure $\mathcal{S}=(S;\,+,\cdot)$ with two binary operations, addition $+$ and multiplication $\cdot$, such that
\begin{itemize}
\item the additive reduct $(S;\,+)$ is a \emph{semilattice} (that is, a commutative and idempotent semigroup),
\item the multiplicative reduct $(S;\,\cdot)$ is a semigroup,
\item multiplication distributes over addition on the left and on the right:
\[
s(t+u)=st+su\ \text{ and }\ (s+t)u=su+tu\ \text{ for all }\ s,t,u\in S.
\]
\end{itemize}
Motivations for studying such structures come from various parts of pure and applied mathematics (e.g., idempotent analysis, tropical geometry, and optimization), where \ais{}s often serve as the foundation for fruitful algebraic approaches; contributions to the conference volumes \cite{Guna98,LiMa:2005} provide diverse examples. Within algebra and its applications to computer science, \ais{}s arise, e.g., in studying binary relations \cite{Jip17}, regular languages \cite{Polak01,Polak03a,Polak03b,Polak04}, endomorphisms of semilattices \cite{JKM09}, group complexity of finite semigroups \cite[Chapter 9]{RhSt09}. Ai-semirings \emph{per se} attract researchers' growing attention as a rich source of rather unexpected examples and challenging problems.

An important family of \ais{}s results from the powerset construction applied to semigroups. The powerset $\mathcal{P}(S)$ of any set $S$ constitutes a semilattice under union; besides, if $(S,\cdot)$ is a semigroup, then one can extend its multiplication to $\mathcal{P}(S)$ by letting
\[
A\cdot B:=\{ab\mid a\in A,\ b\in B\}\ \text{ for all }\ A,B\subseteq S.
\]
It is known and easy to verify that the multiplication defined this way is associative and distributes over union on the left and on the right. Thus, $(\mathcal{P}(S);\,\cup,\cdot)$ is an \ais, called the \emph{power semiring} of the semigroup $(S,\cdot)$.

We address the finite axiomatizability question (aka the Finite Basis Problem) for semiring identities satisfied by power semirings. Recall that a set $\Sigma$ of identities valid in an \ais{} $\mathcal S$ is called an \emph{identity basis} for $\mathcal S$ if all identities holding in $\mathcal S$ can be inferred from $\Sigma$. If $\mathcal S$ admits a finite identity basis, it is \emph{finitely based}; otherwise, $\mathcal S$ is \emph{nonfinitely based}. The Finite Basis Problem for a class of \ais{}s is the issue of determining which \ais{}s from this class are finitely based and which are not. For several classes of \ais{}s, the Finite Basis Problem was studied in the first decade of the 2000s; see \cite{AEI03,AM11} and a series of papers by the first-named author~\cite{Dolinka07,Dolinka09a,Dolinka09b,Dolinka09c} for the main results of that period. A new activity phase has emerged since the beginning of the 2020s when new powerful methods have been developed; see \cite{Vol21,JackRenZhao22,GuVo23a,GuVo23b,RJZL23,WuRenZhao,WuZhaoRen23}.

The first-named author initiated the study of the Finite Basis Problem for power semirings in~\cite{Dolinka09a}. In particular, Problem 6.5 in \cite{Dolinka09a} asked for a description of finite groups whose power semirings are nonfinitely based (provided such groups exist, which was then unknown). In~\cite{GuVo23b}, the second and third-named authors gave a partial solution: by~\cite[Theorem 1.1]{GuVo23b}, the power semiring of any finite nonabelian solvable group is nonfinitely based.

In the present note, we combine techniques from~\cite{Dolinka09a} and~\cite{GuVo23b} to extend the result just stated to power semirings of finite inverse semigroups. Recall that a semigroup $(S,\cdot)$ is said to be \emph{inverse} if for each element $s\in S$, there is a unique $t\in S$ satisfying $sts=s$ and $tst=t$; such $t$ is called the \emph{inverse of $s$} and denoted by $s^{-1}$. (The notation is borrowed from group theory and is justified by the fact that every group is an inverse semigroup in which the unique inverse of any element is nothing but its group inverse.) An inverse semigroup that is a union of its subgroups is called a \emph{Clifford semigroup}.

\begin{theorem}
\label{thm:powerinvsgp}
Let $\mathcal{S}=(S,\cdot)$ be a finite inverse semigroup. Suppose that either $\mathcal{S}$ is not Clifford or all subgroups of $\mathcal{S}$ are solvable and at least one of them is nonabelian. Then the identities of the power semiring $(\mathcal{P}(S);\,\cup,\cdot)$ admit no finite basis.
\end{theorem}

Theorem~\ref{thm:powerinvsgp} essentially widens our knowledge of the Finite Basis Problem for power semirings since inverse semigroups are much more general than groups. In addition, as in~\cite{GuVo23b}, semiring arguments apply \emph{mutatis mutandis} to another natural algebraic structure on power sets. An \emph{involution semigroup} is a semigroup equipped with an \emph{involution}, that is, a unary operation $x\mapsto x^*$ that fulfils the laws $(xy)^*=y^*x^*$ and $(x^*)^*=x$. In each inverse semigroup $(S,\cdot)$, the operation that maps every element to its inverse is known to be an involution. This involution naturally extends to the powerset of $S$ by letting\footnote{A word of warning appears to be in place here: even though we retain the notation and the name `inversion', in general, the subset $A^{-1}$ is not the inverse of $A$ in the sense of the above definition of an inverse element!} $A^{-1}:=\{s^{-1}\mid s\in A\}$ for each $A\subseteq S$. This way we get the \emph{power involution semigroup} $(\mathcal{P}(S);\,\cdot,{}^{-1})$. Combining techniques from~\cite{ADV12} and~\cite{GuVo23b} leads to a generalization of~\cite[Theorem 1.3]{GuVo23b} parallel to the generalization of~\cite[Theorem 1.1]{GuVo23b} provided by Theorem~\ref{thm:powerinvsgp}. Recall that a group is called \emph{Dedekind} if all its subgroups are normal.

\begin{theorem}
\label{thm:invpowerinvsgp}
Let $\mathcal{S}=(S,\cdot)$ be a finite inverse semigroup. Suppose that either $\mathcal{S}$ is not Clifford or all subgroups of $\mathcal{S}$ are solvable and at least one of them is non-Dedekind. Then the identities of the power involution semigroup $(\mathcal{P}(S);\,\cdot,{}^{-1})$ admit no finite basis.
\end{theorem}

The literature has lots of results on the Finite Basis Problem for involution semigroups (see, e.g., Section 1.5 of the recent monograph \cite{LeeBook:23}), but none of the previously known facts apply to the power involution semigroups of finite inverse semigroups.


\section{Powers of finite Clifford semigroups}
\label{sec:clifford}

The proof of our main results exploits some properties of power semigroups $(\mathcal{P}(S);\,\cdot)$ of a finite Clifford semigroup $(S;\,\cdot)$. Recall that a \emph{semilattice} is a semigroup satisfying the laws of commutativity $xy=yx$ and idempotency $x^2=x$. A semigroup $\mS=(S;\,\cdot)$ is said to be a \emph{semilattice of groups} if there exists a homomorphism $\varphi$ from $\mS$ onto a semilattice $\mY=(Y;\,\cdot)$ such that for each $\alpha\in Y$, the set $G_\alpha:=\alpha\varphi^{-1}$ is a subgroup of $\mS$. In this situation, $S=\bigcup_{\alpha\in Y}G_\alpha$ and we say that $\mS$ is the semilattice $\mY$ of the subgroups $G_\alpha$, $\alpha\in Y$. A classic structural result of semigroup theory, Clifford's theorem \cite[Theorem 3]{Clifford:1941} (see also \cite[Theorem 4.11]{Clifford&Preston:1961} or \cite[Theorem 4.2.1]{Howie:1995}), implies that every Clifford semigroup is a semilattice of groups.

A semigroup is called a \emph{block-group} if it satisfies the following quasi-identities:
\begin{gather}
ef=e^2=e\mathrel{\&}fe=f^2=f\to e=f,\label{eq:bg1}\\
ef=f^2=f\mathrel{\&}fe=e^2=e\to e=f.\label{eq:bg2}
\end{gather}
The power semigroups of finite groups are block-groups \cite[Proposition 2.4]{Pin:1980}; see also\cite[Proposition 11.1.2]{Almeida:95}. We extend this to the power semigroups of finite Clifford semigroups.

\begin{proposition}
\label{prop:block-group}
If $\mS=(S;\,\cdot)$ is a finite Clifford semigroup, then the power semigroup $(\mathcal{P}(S);\,\cdot)$ is a block-group.
\end{proposition}

\begin{proof}
By symmetry, it suffices to show that the semigroup $(\mathcal{P}(S);\,\cdot)$ satisfies the implication \eqref{eq:bg1}. Thus, let $E,F\subseteq S$ be such that $EF=E^2=E$ and $FE=F^2=F$; we aim to show that $E=F$. If one of the subsets $E$ or $F$ is empty, then so is the other one, and we are done. Assume that $E,F\ne\varnothing$; then the equalities $E^2=E$ and $F^2=F$ imply that the subsets $E$ and $F$ are closed under multiplication in $\mS$, and so they are subsemigroups in $\mS$.

Let $\mS$ be a semilattice $\mY=(Y;\,\cdot)$ of groups $G_\alpha$, $\alpha\in Y$ and $\varphi\colon\mS\to\mY$ the  homomorphism such that $G_\alpha=\alpha\varphi^{-1}$ for each $\alpha\in Y$. Applying $\varphi$ to the equalities $EF=E$ and $FE=F$ yields
\begin{align*}
  (E\varphi)(F\varphi) & =(EF)\varphi=E\varphi\ \text{ and} \\
  (F\varphi)(E\varphi) & =(FE)\varphi=F\varphi.
\end{align*}
Since the semilattice $\mY$ fulfils the commutative law, we have $(E\varphi)(F\varphi)=(F\varphi)(E\varphi)$, whence $E\varphi=F\varphi$. Clearly, $E\varphi=\{\alpha\in Y\mid E\cap G_\alpha\ne\varnothing\}$ and $F\varphi=\{\alpha\in Y\mid F\cap G_\alpha\ne\varnothing\}$ so that the equality $E\varphi=F\varphi$ means that $E\cap G_\alpha\ne\varnothing$ if and only if $F\cap G_\alpha\ne\varnothing$. For each $\alpha\in E\varphi$, the intersection $E\cap G_\alpha$ is a subsemigroup in the finite group $G_\alpha$, and hence a subgroup since in a finite group, every subsemigroup is a subgroup. Also $F\cap G_\alpha$ is a subgroup in $G_\alpha$ for each $\alpha\in E\varphi$. Obviously, the element-wise product of two subgroups of any groups contains the union of these subgroups. Using this, we see that
\begin{align*}
E=EF&=\left(\bigcup_{\alpha\in E\varphi}(E\cap G_\alpha)\right)\cdot\left(\bigcup_{\alpha\in E\varphi}(F\cap G_\alpha)\right)=\bigcup_{\alpha,\beta\in E\varphi}(E\cap G_\alpha)(F\cap G_\beta)\\
   &\supseteq \bigcup_{\alpha\in E\varphi}(E\cap G_\alpha)(F\cap G_\alpha)\supseteq\bigcup_{\alpha\in E\varphi}\left((E\cap G_\alpha)\cup(F\cap G_\alpha)\right)\\
   &=\left(\bigcup_{\alpha\in E\varphi}(E\cap G_\alpha)\right)\cup\left(\bigcup_{\alpha\in E\varphi}(F\cap G_\alpha)\right)=E\cup F,
\end{align*}
whence $E\supseteq F$. Analogously, we obtain that $F\supseteq E$, and hence, $E=F$.
\end{proof}

To prove the second property of the power semigroups of finite Clifford semigroups we need, we invoke a general fact established by Putcha.
\begin{lemma}[\!\!{\mdseries\cite[Theorem 1]{Putcha:79}}]
\label{lem:putcha}
Let $\mS=(S;\,\cdot)$ be a finite semigroup, $I$ an ideal of $\mS$, and $G$ a subgroup of the power semigroup $(\mathcal{P}(S);\,\cdot)$. Then $G$ has a normal subgroup $N$ such that $N$ is isomorphic to a subgroup of the semigroup $(\mathcal{P}(I);\,\cdot)$ and the quotient $G/N$ is isomorphic to a subgroup of the semigroup $(\mathcal{P}(S/I);\,\cdot)$.
\end{lemma}

\begin{proposition}
\label{prop:solvability}
If $\mS=(S;\,\cdot)$ is a finite Clifford semigroup with solvable subgroups, then all subgroups of the power semigroup $(\mathcal{P}(S);\,\cdot)$ are solvable.
\end{proposition}

\begin{proof}
Let $\mS$ be a semilattice $\mY=(Y;\,\cdot)$ of groups $G_\alpha$, $\alpha\in Y$. We induct on $|Y|$.

If $|Y|=1$, then $\mS$ is a group. It follows from \cite[Theorem 1.1]{MH73}, see also \cite[Section~11.1, item 2)]{Almeida:95}, that each subgroup of the power semigroup of any finite group is isomorphic to a \emph{divisor} (a quotient of a subgroup) of this group. Any divisor of a solvable group is solvable. Hence, if $\mS$ is solvable, so are all subgroups of $(\mathcal{P}(S);\,\cdot)$.

Let $|Y|>1$ and let $\mu$ be a maximal element of the partially ordered the set $(Y;\,\le)$, where the relation $\le$ is defined by $\alpha\le\beta$ for $\alpha,\beta\in Y$ if and only if $\alpha\beta=\alpha$. (It is known and easy to verify that $\le$ is a partial order.) Then the set $I:=S\setminus G_\mu=\bigcup_{\alpha\ne\mu}G_\alpha$ forms an ideal of $\mS$, and the Rees quotient $S/I$ is isomorphic to $G_\mu^0$, the group $G_\mu$ with 0 adjoined. Invoking Lemma~\ref{lem:putcha} for an arbitrary subgroup $G$ of the power semigroup $(\mathcal{P}(S);\,\cdot)$, we get that a normal subgroup $N$ of $G$ is isomorphic to a subgroup of the semigroup $(\mathcal{P}(I);\,\cdot)$ while the quotient $G/N$ is isomorphic to a subgroup $H$ of the semigroup $(\mathcal{P}(G_\mu^0);\,\cdot)$.

The semigroup $(I;\,\cdot)$ is a semilattice of $|Y|-1$ groups, so that the group $N$ is solvable by the inductive assumption. Since any extension of a solvable group by another solvable group is solvable, proving solvability of $G$ reduces to showing that the subgroup $H$ of $(\mathcal{P}(G_\mu^0);\,\cdot)$ is solvable.

The set $\mathcal{P}(G_\mu^0)$ of all subsets of $G_\mu^0$ partitions into two parts: the set of all subsets that omit 0 and the set of all subsets that contain 0. The former set is nothing but $\mathcal{P}(G_\mu)$ and so it forms a subsemigroup in $(\mathcal{P}(G_\mu^0);\,\cdot)$ while the latter set, which we denote by $\mathcal{P}_0(G_\mu^0)$, is easily seen to form an ideal in the subsemigroup of $(\mathcal{P}(G_\mu^0);\,\cdot)$ consisting of all non-empty subsets of $G_\mu^0$. If $H\cap\mathcal{P}_0(G_\mu^0)\ne\varnothing$, then the whole subgroup $H$ is contained in $\mathcal{P}_0(G_\mu^0)$ because any two elements of the subgroup are multiples of each other and any ideal contains all multiples of each of its elements. Thus, $H$ is a subgroup in either $(\mathcal{P}(G_\mu),\cdot)$ or $(\mathcal{P}_0(G_\mu^0),\cdot)$, but these two semigroups are isomorphic under the map that appends 0 to each subset of $G_\mu$. Therefore, $H$ is isomorphic to a subgroup in $(\mathcal{P}(G_\mu),\cdot)$ in any case, and hence, $H$ is solvable as a divisor of the solvable group $G_\mu$ (see the argument in the induction basis above).
\end{proof}

\begin{remark}
\label{rem:margolis}
Proposition~\ref{prop:solvability} also follows from a general result due to Margolis; see \cite[Theorem~7]{Margolis81}. We have chosen to include the above proof because it is based on completely elementary notions, while an alternative approach would require some understanding of Green's relations and pseudovarieties of semigroups.
\end{remark}

\section{Inherently nonfinitely based \ais{}s and involution semigroups}
\label{sec:infb}

The results of Section~\ref{sec:clifford} allow us to extend the approach of~\cite{GuVo23b} to power semirings of finite Clifford semigroups. To handle power semirings of finite inverse semigroups that are not Clifford, we need another tool, namely, the concept of inherent nonfinitely basability. We give the corresponding definitions for the case of \ais{}s.

The class of all \ais{}s satisfying all identities from a given set $\Sigma$ is called the \emph{variety defined by $\Sigma$}. Obviously, the satisfaction of an identity is inherited by forming direct products and taking divisors, which in the semiring situations are homomorphic images of subsemirings. Hence each variety is closed under these two operators, and varieties are characterized by this closure property (the HSP-theorem; see \cite[Theorem 11.9]{BuSa81}).

A variety is \emph{\fb} if it can be defined by a finite set of identities; otherwise it is \emph{\nfb}. Given an \ais{} ${\mS}$, the variety defined by all identities holding in $\mS$ is denoted by $\var\mS$ and called the \emph{variety generated by $\mS$}. By the very definition, $\mS$ and $\var\mS$ are simultaneously finitely or \nfb.

A variety is said to be \emph{locally finite} if each of its finitely generated members is finite. A finite \ais{} is called \emph{inherently \nfb} if it is not contained in any finitely based locally finite variety. The variety generated by a finite \ais{} is locally finite (this is an easy byproduct of the proof of the HSP-theorem; see \cite[Theorem 10.16]{BuSa81}); hence, to prove that a given finite \ais{} $\mS$ is \nfb, it suffices to exhibit an inherently \nfb\ \ais{} in the variety $\var\mS$.

The first example of an inherently \nfb\ \ais\ was the \ais{} of all binary relations on a two-element set; see \cite[Theorem A]{Dolinka09a}. For the use in the present paper, it is convenient to represent binary relations  as Boolean matrices. Recall that a \emph{Boolean} matrix is a matrix with entries $0$ and $1$ only. The addition and multiplication of such matrices are as usual, except that the addition of the entries is defined as $a+b:=\max\{a,b\}$. Let $R_n$ denote the set of all Boolean $n\times n$-matrices. The \ais\ $(R_n;\,+,\cdot)$ is essentially the same as the \ais\ of all binary relations on an $n$-element set subject to the operations of union and relational composition. Thus, we have:
\begin{proposition}
\label{prop:dolinkasemiring}
The \ais{} $(R_2;\,+,\cdot)$ is inherently \nfb.
\end{proposition}

\begin{remark}
\label{rem:0semirings}
Semirings considered in~\cite{Dolinka09a} have, along with addition and multiplication, the nullary operation of taking 0, which is the neutral element for addition. So, strictly speaking, Theorem A in~\cite{Dolinka09a} deals with $(R_2;\,+,\cdot,\left(\begin{smallmatrix} 0&0\\0&0\end{smallmatrix}\right))$ rather than $(R_2;\,+,\cdot)$. However, the proof of this result in~\cite{Dolinka09a} works without a hitch for the latter semiring as well.
\end{remark}

Replacing  ``\ais{}'' with  ``involution semigroup'' in the definitions appearing so far in this section, we arrive at the definition of an inherently \nfb\ involution semigroup: a finite involution semigroup is \emph{inherently \nfb} if it is not contained in any finitely based locally finite variety of involution semigroups.

Matrix semigroups are natural equipped with the involution $M\mapsto M^T$, where $M^T$ is the transpose of the matrix $M$; for Boolean matrices, this involution corresponds to the operation of forming the dual binary relation. We will need the following result parallel to Proposition~\ref{prop:dolinkasemiring}:
\begin{proposition}[\!\!{\mdseries\cite[Theorem 3.14]{ADV12}}]
\label{prop:dolinkainvsemigroup}
The involution semigroup $(R_2;\,\cdot,{}^T)$ is inherently \nfb.
\end{proposition}

The subsemigroup of the semigroup $(R_2;\,\cdot)$ formed by the Boolean $2\times 2$-matrices
\begin{equation}\label{eq:brandt}
\begin{tabular}{cccccc}
$\left(\begin{matrix} 0&0\\0&0\end{matrix}\right)$
&
$\left(\begin{matrix} 0&1\\0&0\end{matrix}\right)$
&
$\left(\begin{matrix} 0&0\\1&0\end{matrix}\right)$
&
$\left(\begin{matrix} 1&0\\0&0\end{matrix}\right)$
&
$\left(\begin{matrix} 0&0\\0&1\end{matrix}\right)$
\end{tabular}
\end{equation}
is known as the 5-\emph{element Brandt semigroup} $\mB_2$. We denote the set of matrices in \eqref{eq:brandt} by $B_2$ so that $\mB_2=(B_2;\,\cdot)$. The semigroup $\mB_2$ is inverse: the inverse of each matrix in $B_2$ is the transpose of this matrix. Hence, the power set of $B_2$ can be endowed with both \ais{} and involution semigroup structures.

\begin{corollary}
\label{cor:infb}
The power semiring $(\mathcal{P}(B_2);\,\cup,\cdot)$ and the power involution semigroup $(\mathcal{P}(B_2);\,\cdot,{}^{-1})$ are inherently \nfb.
\end{corollary}

\begin{proof}
It was observed in~\cite[Corollary 6.3]{Dolinka09a} that $(\mathcal{P}(B_2);\,\cup,\cdot,\varnothing)$ is inherently \nfb{} as an \ais{} with 0, and the proof applies to the ``\ais{}'' part of our corollary as well. The ``involution'' part does not seem to have been explicitly stated so far, but it follows in virtue of the same argument; the argument itself goes back to at least the proof of Lemma~2 in \cite{Margolis81}. For the reader's convenience, we reproduce the argument in a combined form that covers both parts at the same time.

Consider the map $\sigma\colon\mathcal{P}(B_2)\to R_2$ that sends the empty set to $\left(\begin{smallmatrix} 0&0\\0&0\end{smallmatrix}\right)$ and each non-empty subset $A\subseteq B_2$ to the sum of all matrices in $A$. As every Boolean $2\times 2$-matrix is the sum of some matrices in \eqref{eq:brandt}, the map $\sigma$ is onto. It is routine to verify that $\sigma$ respects the three operations we consider on $\mathcal{P}(B_2)$, that is,
for all $A,B\subseteq B_2$,
\[
(A\cup B)\sigma=A\sigma+B\sigma, \qquad (AB)\sigma=A\sigma\cdot B\sigma, \qquad (A^{-1})\sigma=(A\sigma)^T.
\]
Therefore, $(R_2;\,+,\cdot)$ is a homomorphic image of $(\mathcal{P}(B_2);\,\cup,\cdot)$ and $(R_2;\,\cdot,{}^T)$ is a homomorphic image of $(\mathcal{P}(B_2);\,\cdot,{}^{-1})$. Hence, Propositions~\ref{prop:dolinkasemiring} and~\ref{prop:dolinkainvsemigroup} apply.
\end{proof}

\section{Proofs of Theorems~\ref{thm:powerinvsgp} and~\ref{thm:invpowerinvsgp}}
\label{sec:proof}

\setcounter{section}{1}

Recall the statements that we are going to prove:
\begin{theorem}
Let $\mathcal{S}=(S,\cdot)$ be a finite inverse semigroup. Suppose that either $\mathcal{S}$ is not Clifford or all subgroups of $\mathcal{S}$ are solvable and at least one of them is nonabelian. Then the identities of the power semiring $(\mathcal{P}(S);\,\cup,\cdot)$ admit no finite basis.
\end{theorem}

\begin{theorem}
Let $\mathcal{S}=(S,\cdot)$ be a finite inverse semigroup. Suppose that either $\mathcal{S}$ is not Clifford or all subgroups of $\mathcal{S}$ are solvable and at least one of them is non-Dedekind. Then the identities of the power involution semigroup $(\mathcal{P}(S);\,\cdot,{}^{-1})$ admit no finite basis.
\end{theorem}

The formulations of Theorems~\ref{thm:powerinvsgp} and~\ref{thm:invpowerinvsgp} are almost parallel, and so are the proofs of these theorems. We first consider the cases where the same arguments work for power semirings and involution semigroups and then the case where the argumentation diverges.

\medskip

\noindent\textbf{Case 1.} \emph{$\mathcal{S}$ is not a Clifford semigroup.}

\smallskip

In this case, $\mathcal{S}$ possesses an element that does not lie in any subgroup of $\mS$ whence the inverse subsemigroup of $\mS$ generated by this element is not a group. A complete classification of monogenic (that is, generated by a single element) inverse semigroups can be found in \cite[Chapter IX]{Pet84}; it implies that if a monogenic inverse semigroup is not a group, then it has the 5-element Brandt semigroup $\mB_2$ as a homomorphic image. Thus, $\mB_2$ is the image of an inverse subsemigroup $T$, say, of $\mS$ under a homomorphism $\tau$, say. We extend $\tau$ to a map $\tau^\sharp\colon\mathcal{P}(T)\to\mathcal{P}(B_2)$, letting $A\tau^\sharp:=\{a\tau\mid a\in A\}$ for each subset $A\subseteq T$. The map $\tau^\sharp$ obviously preserves the operations of union and element-wise multiplication over subsets. Besides, using the well-known fact that semigroup homomorphisms between inverse semigroups automatically preserve the operation of taking inverse, one easily verifies that $(A^{-1})\tau^\sharp=(A\tau^\sharp)^{-1}$ for all  $A\subseteq T$. We conclude that $(\mathcal{P}(B_2);\,\cup,\cdot)$ and $(\mathcal{P}(B_2);\,\cdot,{}^{-1})$ are homomorphic images of $(\mathcal{P}(T);\,\cup,\cdot)$ and respectively $(\mathcal{P}(T);\,\cdot,{}^{-1})$. The latter objects are in turn substructures in $(\mathcal{P}(S);\,\cup,\cdot)$ and respectively $(\mathcal{P}(S);\,\cdot,{}^{-1})$. Thus, we have $(\mathcal{P}(B_2);\,\cup,\cdot)\in\var(\mathcal{P}(S);\,\cup,\cdot)$ and $(\mathcal{P}(B_2);\,\cdot,{}^{-1})\in\var(\mathcal{P}(S);\,\cdot,{}^{-1})$. In view of Corollary~\ref{cor:infb}, these containments  ensure that the power semiring $(\mathcal{P}(S);\,\cup,\cdot)$ and the power involution semigroup $(\mathcal{P}(S);\,\cdot,{}^{-1})$ are \nfb{} (and even inherently \nfb).\qed

\medskip

\noindent\textbf{Case 2.} \emph{$\mathcal{S}$ is a Clifford semigroup with solvable subgroups and at least one subgroup of $\mS$ is non-Dedekind.}

\smallskip

By Propositions~\ref{prop:block-group} and~\ref{prop:solvability}, the power semigroup $(\mathcal{P}(S);\,\cdot)$ is a block-group with solvable subgroups. Now we invoke two results from~\cite{GuVo23b} that give conditions for the absence of a finite identity basis for an \ais{} or an involution semigroup whose multiplicative reduct is a block-group with solvable subgroups. These results involve the \ais{} and the involution semigroup whose multiplicative reduct is the 6-\emph{element Brandt monoid} $\mathcal{B}_2^1$ which we now define. The monoid $\mathcal{B}_2^1=(B_2^1,\cdot)$ is the subsemigroup of the semigroup $(R_2;\,\cdot)$ of all Boolean $2\times 2$-matrices that one gets by adjoining the identity $2\times 2$-matrix $\left(\begin{smallmatrix} 1&0\\0&1\end{smallmatrix}\right)$ to the five matrices in \eqref{eq:brandt}. The semigroup $\mB_2^1$ is inverse and so can be considered as an involution semigroup. Besides, $B_2^1$ is known to admit a unique addition $+$ such that $(B_2^1;\,+,\cdot)$ becomes an \ais; the addition is nothing but the Hadamard (entry-wise) product of matrices: $(a_{ij})+(b_{ij}):=(a_{ij}b_{ij})$.

\setcounter{section}{4}
\setcounter{theorem}{0}

\begin{proposition}
\label{prop:gusevaisemiring}
A finite \ais{} $\mathcal{T}=(T,+,\cdot)$ whose multiplicative reduct is a block-group with solvable subgroups is \nfb{} whenever the variety $\var\mathcal{T}$ contains the \ais\ $(B_2^1;\,+,\cdot)$.
\end{proposition}

\begin{proposition}
\label{prop:gusevinvsemigroup}
A finite involution semigroup $\mathcal{T}=(T,\cdot,{}^*)$ whose multiplicative reduct is a block-group with solvable subgroups is \nfb{} whenever the variety $\var\mathcal{T}$ contains the inverse semigroup $(B_2^1;\,\cdot,{}^{-1})$.
\end{proposition}

Proposition~\ref{prop:gusevaisemiring} is a reformulation of Theorem 4.2 in \cite{GuVo23b} and Proposition~\ref{prop:gusevinvsemigroup} follows by combining Proposition~3.9 and Theorem~5.1 of~\cite{GuVo23b}.

To be able to invoke Propositions~\ref{prop:gusevaisemiring} and~\ref{prop:gusevinvsemigroup}, we have to check that $\var(\mathcal{P}(S);\,\cup,\cdot)$ contains the \ais\ $(B_2^1;\,+,\cdot)$ and $\var(\mathcal{P}(S);\,\cdot,{}^{-1})$ contains the inverse semigroup $(B_2^1;\,\cdot,{}^{-1})$. We will do this at the same time, combining arguments from the proofs of Theorems~1.1 and~1.3 in~\cite{GuVo23b}.

Let $G$ be a subgroup of $\mathcal{S}$ that is not Dedekind. Then, some subgroup $H$ of $G$ is not normal. Fix an element $g\in G$ with $g^{-1}Hg\ne H$. Denote by $E$ the singleton subgroup of $\mathcal{G}$, let $J:=\{A\subseteq G\mid |A|>|H|\}$, and consider the following subset of $\mathcal{P}(S)$:
\[
B:=\{E,\,H,\,g^{-1}H,\,Hg,\,g^{-1}Hg\}\cup J.
\]
Notice that all subsets in $\{E,\,H,\,g^{-1}H,\,Hg,\,g^{-1}Hg\}$ are distinct. It is obvious for all pairs of the subsets except, perhaps, for $g^{-1}H\ne Hg$. To verify this inequality, observe that $g^{-1}H\cdot Hg=g^{-1}Hg$ while $Hg\cdot g^{-1}H=H$, so that the assumption $g^{-1}H=Hg$ forces the equality $g^{-1}Hg=H$ that contradicts the choice of $g$.

It is straightforward that the set $B$ is closed under unions and element-wise products as well as under element-wise inversion of its elements. Hence, $(B;\,\cup,\cdot)$ is a subsemiring of $(\mathcal{P}(S);\,\cup,\cdot)$ while $(B;\,\cdot,{}^{-1})$ is an involution subsemigroup of $(\mathcal{P}(S);\,\cdot,{}^{-1})$. Define a map $B\to B_2^1$ via the following rule:
\[
\begin{tabular}{cccccc}
$A\in J$ & $E$ & $Hg$ & $g^{-1}H$ & $H$ & $g^{-1}Hg$\\
$\downarrow$ & $\downarrow$ & $\downarrow$ & $\downarrow$ & $\downarrow$ & $\downarrow$ \\
$\left(\begin{matrix} 0&0\\0&0\end{matrix}\right)$
&
$\left(\begin{matrix} 1&0\\0&1\end{matrix}\right)$
&
$\left(\begin{matrix} 0&1\\0&0\end{matrix}\right)$
&
$\left(\begin{matrix} 0&0\\1&0\end{matrix}\right)$
&
$\left(\begin{matrix} 1&0\\0&0\end{matrix}\right)$
&
$\left(\begin{matrix} 0&0\\0&1\end{matrix}\right)$
\end{tabular}.
\]
The map can be routinely verified to respect the three operations we consider on $B$. Therefore, the \ais\ $(B_2^1;\,+,\cdot)$ is a homomorphic image of $(B,\cup,\cdot)$ and the inverse semigroup $(B_2^1;\,\cdot,{}^{-1})$ is a homomorphic image of $(B;\,\cdot,{}^{-1})$. Thus, $(B_2^1;\,+,\cdot)\in\var(\mathcal{P}(S);\,\cup,\cdot)$ and $(B_2^1;\,\cdot,{}^{-1})\in\var(\mathcal{P}(S);\,\cdot,{}^{-1})$. Hence Propositions~\ref{prop:gusevaisemiring} and~\ref{prop:gusevinvsemigroup} apply.\qed

\medskip

Since Cases 1 and 2 exhaust the premise of Theorem~\ref{thm:invpowerinvsgp}, this theorem is now proved. For Theorem~\ref{thm:powerinvsgp}, it remains to consider the following:

\smallskip

\noindent\textbf{Case 3.} \emph{$\mathcal{S}$ is a Clifford semigroup with Dedekind subgroups and at least one subgroup of $\mS$ is nonabelian.}

\smallskip

Let $G$ be a nonabelian subgroup of $\mathcal{S}$. Since ${G}$ is Dedekind, it has a subgroup isomorphic to the 8-element quaternion group; see \cite[Theorem 12.5.4]{Hall:1959}. This subgroup is isomorphic to a subgroup in the semigroup $(\mathcal{P}(S);\,\cdot)$, and by~\cite[Theorem~6.1]{JackRenZhao22} every \ais\ whose multiplicative reduct has a nonabelian nilpotent subgroup is \nfb. We see that $(\mathcal{P}(G);\,\cup,\cdot)$ has no finite identity basis in this case as well.\qed

\begin{remark}
\label{rem:with0}
It is quite natural to consider $(\mathcal{P}(S);\,\cup,\cdot,\varnothing)$ as an algebra with an extra nullary operation of taking the constant $\varnothing$; this was actually the setting adopted by the first-named author in~\cite{Dolinka09a}. The conclusion of Theorem~\ref{thm:powerinvsgp} persists in this setting. In the argument of Case 1, one should substitute $(\mathcal{P}(B_2);\,\cup,\cdot,\varnothing)$ for $(\mathcal{P}(B_2);\,\cup,\cdot)$ and refer to~\cite[Corollary 6.3]{Dolinka09a} instead of Corollary~\ref{cor:infb}. For Case 2, the fact that the proof of Proposition~\ref{prop:gusevaisemiring} works for semirings with multiplicative zero treated as a constant was explicitly recorded in \cite[Remark 3]{GuVo23b}. For Case 3, one should refer to Theorem~7.6 from \cite{JackRenZhao22} that transfers the result of~\cite[Theorem~6.1]{JackRenZhao22} to semirings with multiplicative zero treated as a constant.
\end{remark}

\smallskip

\textbf{Acknowledgements.} The authors thank Edmond W. H. Lee for valuable remarks.

\end{document}